\documentclass[11pt]{amsart}
\usepackage{amssymb}
\newcommand{\R}{\mathbb{R}}
\newcommand{\N}{\mathbb{N}}
\newcommand{\Z}{\mathbb{Z}}
\newcommand{\C}{\mathbb{C}}

\newcommand{\F}{\mathcal{F}}
\newcommand{\dq}{\partial_q}

\def\endproof{\quad \hfill$\blacksquare$\vspace{0.15cm}\\}
\newcommand{\ds}{\displaystyle}

\newtheorem{defin}{Definition}
\newtheorem{propo}{ Proposition}
\newtheorem{lemme}{Lemma}

\newtheorem{theorem}{Theorem}

\begin{document}
\title [$q$-Dunkl transform on the line ]
 { $q$-analogue of the Dunkl transform on the real line}
\author[N. Bettaibi \& R. Bettaieb]{ N\'eji Bettaibi\quad  \& \quad Rym  H. Bettaieb}
\address{N. Bettaibi, Institut
Pr\'eparatoire  aux \'Etudes d'Ing\'enieur de Monastir, 5000
Monastir, Tunisia.}
 \email{neji.bettaibi@ipein.rnu.tn}

\address{R. Bettaieb, Institut
Pr\'eparatoire  aux \'Etudes d'Ing\'enieur de Monastir, 5000
Monastir, Tunisia.}
 \email{rym.bettaieb@yahoo.fr}

\begin{abstract} In this paper, we consider a $q$-analogue of the Dunkl operator on
$\R$, we  define and study its associated Fourier transform which
is a  $q$-analogue of  the Dunkl transform. In addition to
several properties, we establish   an inversion formula and prove
a Plancherel theorem for this $q$-Dunkl transform. Next, we study
the $q$-Dunkl intertwining operator and its dual via the
$q$-analogues of the Riemann-Liouville and Weyl transforms. Using
this dual intertwining operator, we provide a relation between the
$q$-Dunkl transform and the $q^2$-analogue Fourier transform
introduced and studied in
 \cite{Rubin, Ru}.
\end{abstract}
 \maketitle
\section{Introduction}
The Dunkl operator on $\R$ of index $\ds
\left(\alpha+\frac{1}{2}\right)$ associated with the reflection
group $\Z_2$ is  the differential-difference operator
$\Lambda_\alpha$ introduced by C. F. Dunkl in \cite{d1} by
\begin{equation}\label{D1}
\Lambda_\alpha
(f)(x)=\frac{df(x)}{dx}+\left(\alpha+\frac{1}{2}\right)\frac{f(x)-f(-x)}{x},\quad
\alpha \geq -\frac{1}{2}.
\end{equation}
These operators  are very important in pure mathematics and
physics. They provide a useful tool in the study of special
functions with root systems \cite{d2, Dei}  and they are closely
related to certain representations of degenerate affine Heke
algebras \cite{Che, Op}, moreover the commutative algebra
generated by these operators has been used in the study of certain
exactly solvable models of quantum mechanics, namely the
Calogero-Suterland-Moser models, which deal with systems of
identical particles in a one dimensional space \cite{Lap, Kakei}.

In \cite{d3}, C. F. Dunkl has introduced and studied a Fourier
transform associated with the operator $\Lambda_\alpha$, called
Dunkl transform, but the basic results such as inversion formula
and Placherel theorem were established later by M. F. E. de Jeu in
\cite{j1, j2}.

 C. F. Dunkl has proved in \cite{d2} that there
 exists a linear isomorphism $V_\alpha$, called the Dunkl
 intertwining operator, from the space of polynomials on $\R$ of degree $n$ onto itself,
 satisfying the transmutation relation
 \begin{equation}\label{trans}
\Lambda_\alpha V_\alpha=V_\alpha \frac{d}{dx},\qquad
V_\alpha(1)=1.
 \end{equation}
Next, K. Trim\`eche has proved in \cite{trim} that the operator
$V_\alpha$ can be extended to a topological isomorphism from
$\mathcal{E}(\R)$, the space of $C^\infty$-functions on $\R$, onto
itself satisfying the relation (\ref{trans}).

The goal of this paper is to provide a similar construction for a
$q$-analogue context. The analogue transform we employ to make our
construction is based on some $q$-Bessel functions and
orthogonality results from \cite{KS}, which have important
applications to $q$-deformed mechanics. The $q$-analogue of the
Bessel operator and the Dunkl operator are defined in terms of the
$q^2$-analogue differential operator, $\partial_q$, introduced in
\cite{Ru}.\\

This paper is organized as follows: In Section 2, we present some
preliminaries results and notations that will be useful in the
sequel. In Section 3, we establish some results associated with
the $q$-Bessel transform and study the  $q$-Riemann-Liouville and
the $q$-Weyl  operators. In Section 4, we introduce and study a
$q$-analogue of the Dunkl operator (\ref{D1}) and  we deal with
its eigenfunctions by  giving some of their properties and
providing for them a $q$-integral representations of Mehler type
as well as an orthogonality relation. In section 5, we define and
study the $q$-Dunkl intertwining operator and its dual via the
$q$-Riemann-Liouville and the $q$-Weyl transforms. Finally, in
Section 6,  we study the Fourier transform associated with the
$q$-Dunkl operator ($q$-Dunkl transform),  we establish an
inversion formula,  prove a Plancherel theorem  and we provide a
relation between the $q$-Dunkl transform and the $q^2$-analogue
Fourier transform (see \cite{Rubin, Ru}).

\section{Notations and preliminaries}
For the convenience of the reader, we provide in this section a
summary of the mathematical notations and definitions used in this
paper. We refer the reader to the general references \cite{GR} and
\cite{KC}, for the definitions, notations and properties of the
$q$-shifted factorials and the $q$-hypergeometric functions.
Throughout this paper, we assume $q\in]0, 1[$ and we  denote $\ds
\R_q=\{\pm q^n~~:~~n\in\Z\}$,  $\ds \R_{q,+}=\{q^n~~:~~n\in\Z\}$.\\

\subsection{Basic symbols}
For $x \in \C $, the $q$-shifted factorials are defined by
\begin {equation}
 (x;q)_0=1;~~ (x;q)_n = \ds\prod _{k=0}^{n-1}(1-xq^k),~~
 n=1,2,...;~~  (x;q)_\infty = \ds\prod _{k=0}^\infty(1-xq^k).
\end {equation}
We also denote
 \begin {equation}
 [x]_q={{1-q^x}\over{1-q}},~~~~~ x\in \C\quad {\rm and}\quad [n]_q! ={{(q;q)_n}\over
 {(1-q)^n}},  ~~~~~~ n\in \N.
\end {equation}

\subsection{ Operators and elementary special  functions}~~\\
 The  $q$-Gamma function is given by (see  \cite{Jac} )
$$
\Gamma_q (x) ={(q;q)_{\infty}\over{(q^x;q)_{\infty}}}(1-q)^{1-x} ,
~~ ~~x\neq 0, -1 , -2 ,...
$$
It satisfies the  following relations \begin {equation}
\label{gam1} \Gamma_q (x+1) =[x]_q \Gamma_q (x) ,~~ \Gamma_q (1) =
1~~ and \lim_{q\longrightarrow 1^-}\Gamma_q (x) = \Gamma(x) ,
\Re(x) >0.
\end {equation}
 The $q$-trigonometric functions $q$-cosine and $q$-sine are
defined by ( see \cite {Rubin, Ru})
\begin {equation}\label{cos}
 \cos(x;q^2)=\sum_{n=0}^{\infty}(-1)^nq^{n(n+1)}
{{x^{2n}}\over{[2n]_q!}}\quad{\rm ,
}\quad\sin(x;q^2)=\sum_{n=0}^{\infty}(-1)^nq^{n(n+1)}
{{x^{2n+1}}\over{[2n+1]_q!}}.
\end {equation}
 The $q$-analogue exponential function is given by ( see \cite {Rubin, Ru})
\begin {equation}\label{exp}
e(z;q^2)=\cos(-iz;q^2)+i\sin(-iz;q^2).
\end {equation}
These three functions are absolutely convergent for all $z$ in the
plane and when $q$ tends to 1 they tend to the corresponding
classical ones  pointwise and uniformly on compacts.\\
Note that we have for all $x \in \R_q$ (see \cite{Rubin})
$$|\cos(x;q^2)|\leq \frac{1}{(q;q)_\infty},\quad  |\sin(x;q^2)|\leq
\frac{1}{(q;q)_\infty},$$ and
\begin{equation}\label{iexp}
    |\ e(ix;q^2)|\leq \ds \frac{2}{(q;q)_\infty}.
\end{equation}
The $q^2$-analogue differential operator is ( see \cite {Rubin,
Ru})
\begin {equation}\partial_q(f)(z)=\left\{\begin{array}{cc}
                   \ds\frac{f\left(q^{-1}z\right)+f\left(-q^{-1}z\right)-f\left(qz\right)+f\left(-qz\right)-2f(-z)}{2(1-q)z}
 & if~~z\neq 0 \\
                   \ds\lim_{x\rightarrow 0}\partial_q(f)(x)\qquad({\rm in}~~ \R_q) & if~~z= 0. \\
                 \end{array}\right.
\end {equation}
Remark that  if $f$ is differentiable at $z$, then $\ds
\lim_{q\rightarrow1}\partial_q(f)(z)=f'(z)$.\\
A repeated application of the $q^2$-analogue differential operator
$n$ times is denoted by:$$ \partial_q^0f=f,\quad
\partial_q^{n+1}f=\partial_q(\partial_q^nf).$$
 The following lemma lists some useful computational
properties of $\dq$, and reflects the sensitivity of this operator
to parity of its argument. The proof is straightforward.
\begin{lemme}\label{ld}~~\\
1) $\dq \sin(x;q^2)=\cos(x;q^2)$, $\dq \cos(x;q^2)=-\sin(x;q^2)$
and $\dq e(x;q^2)=e(x;q^2)$.\\
2)  For all function $f$ on $\R_q$, $\ds \dq
f(z)=\frac{f_e(q^{-1}z)-f_e(z)}{(1-q)z}+\frac{f_o(z)-f_o(qz)}{(1-q)z}.$\\
3) For two functions $f$ and $g$ on $\R_q$, we have \\
  \indent $\centerdot$ if $f$ even and $g$ odd $$
  \partial_q(fg)(z)=q\partial_q(f)(qz)g(z)+
  f(qz)\partial_q(g)(z) =\partial_q(g)(z))f(z)+
  qg(qz)\partial_q(f)(qz);$$
\indent $\centerdot$ if $f$ and $g$ are even   $$
  \partial_q(fg)(z)=\partial_q(f)(z)g(q^{-1}z)+
  f(z)\partial_q(g)(z).$$
\end{lemme}
Here,  for a function $f$ defined on $\R_q$, $f_e$ and  $f_o$ are
its  even and odd parts respectively.\\
 The $q$-Jackson integrals are defined
by (see \cite {Jac})  {\small\begin {equation} \label{int1}
\int_0^a{f(x)d_qx} =(1-q)a \sum_{n=0}^{\infty}q^nf(aq^n),\quad
\int_a^b{f(x)d_qx} =(1-q)\sum_{n=0}^{\infty}q^n\left(b f(bq^n)-a
f(aq^n)\right),
\end {equation}}
\begin {equation} \label{int2}
\int_0^{\infty}f(x)d_qx
=(1-q)\sum_{n=-\infty}^{\infty}q^nf(q^n),\quad
\int_{-\infty}^{\infty}f(x)d_qx
=(1-q)\sum_{n=-\infty}^{\infty}q^n\left[f(q^n)+f(-q^n)\right],
\end {equation}
provided the sums converge absolutely. In particular, for
$a\in\R_{q,+}$, \begin {equation} \label{int4}
\int_{a}^{\infty}{f(x)d_qx} =(1-q)a \sum_{n=-\infty}^{-1}q^nf(a
q^n),
\end {equation}
The following simple result, giving  $q$-analogues of the
integration by parts theorem, can be verified by direct
calculation.
\begin{lemme}\label{ppar}~~\\
1) For $a>0$, if $\ds \int_{-a}^a (\dq f)(x)g(x)d_qx$ exists, then
{\small \begin{equation} \int_{-a}^a (\dq
f)(x)g(x)d_qx=2\left[f_e(q^{-1}a)g_o(a)+f_o(a)g_e(q^{-1}a)\right]-\int_{-a}^a
f(x)(\dq g)(x)d_qx.
\end{equation}}
2) If $\ds \int_{-\infty}^\infty (\dq f)(x)g(x)d_qx$ exists,
\begin{equation}
\int_{-\infty}^\infty (\dq f)(x)g(x)d_qx=-\int_{-\infty}^\infty
f(x)(\dq g)(x)d_qx.
\end{equation}
\end{lemme}
\subsection{Sets and spaces}~~\\
By the use of the $q^2$-analogue differential operator $\dq$, we
note:\\
$ \bullet$ $ \mathcal{E}_{q}(\R_q)$ the space of   functions $f$
defined on $\R_q$, satisfying
  $$\forall n\in\N,~~a \geq 0,
 ~~~~P_{n,a}(f) = \sup\left\{ |\partial _q^kf(x)|; 0\leq k\leq n ; x \in
 [-a,a]\cap \R_q\right\}<\infty$$ and  $$ \lim_{x\rightarrow 0}\partial _q^nf(x)\quad ({\rm in}\quad
 \R_q)\qquad {\rm
 exists}.$$
 We provide it with the topology  defined by the semi
 norms $P_{n,a}.$ \\
 $ \bullet$ $ \mathcal{E}_{\ast ,q}(\R_q)$ the subspace of
$\mathcal{E}_{q}(\R_q)$ constituted of even
functions. \\
$ \bullet$ $ \mathcal{S}_{q}(\R_q)$ the space of  functions $f$
defined  on $\R_q$ satisfying
$$\forall n,m\in\N,~~~~P_{n,m,q}(f)=\sup_{x\in\R_q}\mid x^m \dq^nf(x)\mid<+\infty$$
and $$ \lim_{x\rightarrow
0}\partial _q^nf(x)\quad ({\rm in}\quad
 \R_q)\qquad {\rm
 exists}.$$
  $\bullet$ $ \mathcal{S}_{\ast ,q}(\R_q)$ the subspace of
$\mathcal{S}_{q}(\R_q)$ constituted of even
functions. \\
$\bullet$ $ \mathcal{D}_{q}(\R_q)$ the space of   functions defined on $\R_q$ with compact supports.\\
$\bullet$ $ \mathcal{D}_{\ast ,q}(\R_q)$ the subspace of
$\mathcal{D}_{q}(\R_q)$ constituted of even
functions. \\
Using the $q$-Jackson integrals, we note for $p>0 $ and  $\alpha \in \R$,\\
$\bullet \ds L_q^p(\R_q)=\left\{f:
\|f\|_{p,q}=\left(\int_{-\infty}^{\infty}|f(x)|^pd_qx\right)^{\frac{1}{p}}<\infty
\right\},$ \\
$\bullet$ $ \ds L_q^p(\R_{q,+})=\left\{f:
\|f\|_{p,q}=\left(\int_{0}^{\infty}|f(x)|^pd_qx\right)^{\frac{1}{p}}<\infty
\right\},$ \\
$ \bullet$ $ \ds L_{\alpha ,q}^p(\R_q)=\left\{f:
\|f\|_{p,\alpha,q}=\left(\int_{-\infty}^{\infty}|f(x)|^p
 |x|^{2\alpha +1}
 d_qx\right)^{\frac{1}{p}}<\infty \right\}, $\\
 $ \bullet$ $ \ds L_{\alpha ,q}^p(\R_{q,+})=\left\{f:
\|f\|_{p,\alpha,q}=\left(\int_{0}^{\infty}|f(x)|^p
 x^{2\alpha +1} d_qx\right)^{\frac{1}{p}}<\infty \right\}, $\\
$ \bullet$ $ \ds L_q^\infty(\R_q)=\left\{f:
\|f\|_{\infty,q}=\sup_{x\in\R_q}|f(x)|<\infty\right\}$, \\
$ \bullet$ $ \ds L_q^\infty(\R_{q,+})=\left\{f:
\|f\|_{\infty,q}=\sup_{x\in\R_{q,+}}|f(x)|<\infty\right\}. $ \\

\subsection{ $q^2$-Analogue Fourier transform}
R. L. Rubin defined in \cite{Ru} the $q^2$-analogue Fourier
transform as
\begin {equation}\label{fou}
\widehat{f}(x;q^2)=K \int_{-\infty}^\infty f(t)e(-itx;q^2)d_qt,
\end {equation}
where $K = \ds \frac{(1+q)^{\frac{1}{2}}}{2\Gamma_{q^2}
\left(\frac{1}{2}\right)}.$\\
 Letting $q\uparrow 1$ subject to the condition
\begin{equation}\label{q}
\frac{Log(1-q)}{Log(q)}\in 2\Z,
\end{equation} gives,
at least formally, the classical Fourier transform. In the
remainder of this paper, we assume that the condition (\ref{q})
holds.

 It was shown in \cite{Ru} that $\widehat{f}(.;q^2)$ verifies the following
 properties:\\
 1) If $f(u),~~uf(u)\in L_q^1(\R_q)$, then $\ds \dq\left(~\widehat{f}~\right) (x;q^2)
 =(-iuf(u))\widehat{}(x;q^2).$\\
  2) If $f, ~~\dq f\in L_q^1(\R_q)$, then $\ds \left(\dq
 f\right)~\widehat{}~(x;q^2)=ix\widehat{f}~(x;q^2)$.\\
 3) $\widehat{f}~(.;q^2)$ is an isomorphism from $L_q^2(\R_q)$ onto
 itself.  For $f\in L_q^2(\R_q)$, we have\\
 $\ds \forall x\in\R_q,~~  \left(\widehat{f}  \right)^{-1}(x;q^2)=\widehat{f}(-x;q^2)$ and $\ds
\|\widehat{f}~(.;q^2)\|_{2,q}=\|f\|_{2,q}$.

\section{$q$-Bessel Fourier Transform}
The normalized $q$-Bessel function is defined by
\begin{equation}\label{j}
j_\alpha(x;q^2) =  \ds \sum_{n=0}^{+\infty} (-1)^n
\frac{\Gamma_{q^2}(\alpha+1)q^{n(n+1)}}{\Gamma_{q^2}(\alpha+n+1)
\Gamma_{q^2}(n+1)}\left( \frac{x}{1+q} \right)^{2n}.
\end{equation}
Note that  we have
\begin{equation}\label{jJ} j_\alpha(x;q^2) = (1-q^2)^\alpha \Gamma_{q^2}(\alpha+1)
\left((1-q)x\right)^{-\alpha}  J_\alpha((1-q)x ; q^2),
\end{equation}
where \begin{equation}J_\alpha(x ; q^2)= \ds \frac{x^\alpha
(q^{2\alpha +2};q^2)_\infty}{(q^2;q^2)_\infty} ._1\varphi
_1(0;q^{2\alpha +2};q^2,q^2x^2)
\end{equation}
is the Jackson's third $q$-Bessel function.\\
 Using the relations (\ref{j}) and  (\ref{cos}), we
obtain
\begin {equation}\label{jcos}j_{-\frac{1}{2}}(x;q^2) = \cos(x;q^2),\end {equation}
\begin {equation}\label{jsin}j_{\frac{1}{2}}(x;q^2) =\frac{\sin(x;q^2)}{x} \end
{equation} and
\begin {equation}\label{dqj}
\partial_q j_\alpha(x;q^2)=-\frac{x}{[2\alpha +2]_q}j_{\alpha +1}(x;q^2)
.\end {equation} In \cite{NW}, the authors proved the following
estimation.
\begin{lemme}\label{asympj}
For $\ds\alpha \geq -\frac{1}{2}$~~and  $x\in \R_{q},$
\begin{eqnarray*}
 \bullet ~~~~|j_\alpha(x;q^2)| & \leq & \ds\frac{(-q^2;q^2)_\infty (-q^{2\alpha +1};q^2)_\infty}
  {(q^{2\alpha +1};q^2)_\infty} \left \{ \begin{array}{cc}
    \\                                     1, & {\rm if}~~ |x| \leq \frac{1}{1-q} \\
                                           q^{\left(\frac{Log(1-q)|x|}{Log
                                           q}\right)^2},
                                           & {\rm if}~~ |x| \geq \frac{1}{1-q} \\
                                         \end{array}\right.
\end{eqnarray*} ~~
\indent \quad \ $\bullet$ ~~for~~ all ~~$v \in \R,$
  $
  j_\alpha(x;q^2) = o(x^{-v})~~ as~~ |x| \longrightarrow +\infty \quad ({\rm in}\quad \R_q).
$
\end{lemme}
As a consequence of  the previous lemma and the relation
(\ref{dqj}), we have for $\ds\alpha \geq -\frac{1}{2}$,
$$j_\alpha(.;q^2)\in
\mathcal{S}_{*,q}(\R_q).$$
 With the same technique used in \cite{BZ}, we can prove
that
 for~~ $\ds \alpha >-\frac{1}{2}$,  $j_\alpha(.;q^2)$ has  the
following $q$-integral representation of Mehler type
\begin {equation}
j_\alpha(x;q^2) = C(\alpha ;q^2) \ds\int_0^1W_\alpha (t;q^2)
\cos(xt;q^2) d_qt,\end{equation} where
\begin {equation}\label{cq}
 C(\alpha;q^2) = (1+q) \frac{\Gamma _{q^2}(\alpha +1)}{\Gamma
_{q^2}(\frac{1}{2})\Gamma _{q^2}(\alpha +\frac{1}{2})}
 \end {equation}
 and
 \begin{equation}\label{w} W_\alpha (t;q^2)=
\frac{(t^2q^2;q^2)_\infty}{(t^2q^{2\alpha +1};q^2)_\infty}.
\end{equation}
{\bf Remark:} Since the functions  $  W_\alpha (.;q^2)$ and~~
$\cos (.;q^2)$ are even and $\sin (.;q^2)$ is odd, we can write
for~~ $\ds \alpha >-\frac{1}{2},$
\begin {equation}\label{jmeh}j_\alpha(x;q^2) = \frac{1}{2} C(\alpha ;q^2)
\ds\int_{-1}^1W_\alpha (t;q^2) e(-ixt;q^2) d_qt.\end{equation} In
particular, using the inequality (\ref{iexp}), we obtain
\begin {equation}\label{ij}|j_\alpha(x;q^2)| \leq \frac{2}{(q;q)_\infty} ,
\forall x\in \R_{q}.\end {equation}
\begin{propo} \label{ort} For $x,y \in \R_{q,+},$ we have
\begin{equation}\label{jort} (xy)^{\alpha+1} \ds\int_0^{+\infty}
j_\alpha(xt;q^2) j_\alpha(yt;q^2) t^{2\alpha +1} d_qt =
\frac{(1+q)^{2\alpha} \Gamma _{q^2}^2 (\alpha +1)}{(1-q)}\delta
_{x,y}.\end{equation}\end{propo}
\begin{proof} The result follows from the relation (\ref{jJ})
and the orthogonality relation of the Jackson's third $q$-Bessel
function $J_\alpha (.;q^2)$ proved in \cite{KS}.
\end {proof}Using the same technique as in \cite{BZ}, one can
prove the following result.
\begin{propo}\label{pj} For $\lambda \in
\C $, the function $j_\alpha(\lambda x;q^2)$ is the unique even
solution of the problem \begin{equation} \left\{\begin{array}{c}
                          \triangle_{\alpha,q} f(x) = -\lambda^2
f(x), \\
                          f(0) = 1 , \\
                       \end{array}\right. \end{equation}

where $  \triangle_{\alpha,q} f(x) = \ds \frac{1}{|x|^{2\alpha
+1}}
\partial _q[|x|^{2\alpha +1}\partial _q f(x)]$.
\end{propo}
\begin{defin} The $q$-Bessel Fourier transform is defined for $f\in
L_{\alpha,q}^1(\R_{q,+}) , $ by \begin {equation}\label{FB}
\mathcal{F}_{\alpha ,q}(f)(\lambda) = c_{\alpha ,q}
\ds\int_0^\infty f(x)j_\alpha(\lambda x;q^2)x^{2\alpha +1}d_qx
\end {equation} where  \begin {equation}\label{c} c_{\alpha ,q} =
\frac{(1+q)^{-\alpha}}{\Gamma_{q^2}(\alpha+1)}.\end {equation}
\end{defin}
 Letting $q\uparrow 1$ subject to the condition
(\ref{q}),  gives, at least formally, the classical Bessel-Fourier
transform.\\ Some properties of the $q$-Bessel Fourier transform
are given in the following result.
\begin{propo}
 1) For  $f\in L_{\alpha,q}^1(\R_{q,+})$, we have $\mathcal{F}_{\alpha
 ,q}(f)\in L_{q}^\infty(\R_{q,+})$ and $$\|\mathcal{F}_{\alpha
 ,q}(f)\|_{\infty , q}\leq \frac{2c_{\alpha ,q}}{(q;q)_\infty}\|f\|_{1,q}.$$
 2)For  $f, g \in L_{\alpha,q}^1(\R_{q,+})$, we have
 \begin{equation}\label{sym}
    \int_0^\infty f(x)\mathcal{F}_{\alpha ,q}(g)(x)x^{2\alpha
    +1}d_qx =  \ds\int_0^\infty \mathcal{F}_{\alpha ,q}(f)(\lambda) g(\lambda)\lambda^{2\alpha
    +1}d_q\lambda.
\end{equation}
 3) If $f$ and $\triangle_{\alpha,q} f$ are in
 $L_{\alpha,q}^1(\R_{q,+})$, then $$\mathcal{F}_{\alpha
 ,q}(\triangle_{\alpha,q}f)(\lambda)=-\lambda^2\mathcal{F}_{\alpha
 ,q}(f)(\lambda).$$
 4) If $f$ and $x^2 f$ are in $L_{\alpha,q}^1(\R_{q,+})$, then
 $$\triangle_{\alpha,q} (\mathcal{F}_{\alpha ,q}(f))= - \mathcal{F}_{\alpha ,q}(x^2
 f).$$

\end{propo}
\proof 1) follows from the definition of $\mathcal{F}_{\alpha
 ,q}$ and the relation (\ref{ij}).\\
2) Let  $f,g \in L_{\alpha ,q}^1 (\R_{q,+}).$ \\ Since for all
$\lambda ,x \in \R_{q,+} $, we have $\ds \mid j_\alpha(\lambda
x;q^2)\mid \leq \frac{2}{(q;q)_\infty}$,  then
\begin{eqnarray*}
  \int_{0}^{+\infty}\ds\int_{0}^{+\infty} \mid f(x)
g(\lambda ) j_\alpha(\lambda x;q^2)| x ^{2\alpha +1} \lambda
^{2\alpha +1}d_q x d_q \lambda  & \leq & \frac{2}{(q;q)_\infty} \|
f\|_{1,\alpha ,q} \| g\|_{1,\alpha ,q}<\infty .
\end{eqnarray*}
So, by the Fubini's theorem, we can exchange the order of the
$q$-integrals and obtain,
\begin{eqnarray*}
  \int_0^\infty f(x)\mathcal{F}_{\alpha ,q}(g)(x)x^{2\alpha
    +1}d_qx &=&\int_{0}^{+\infty}\int_{0}^{+\infty}  f(x)
g(\lambda ) j_\alpha(\lambda x;q^2) x ^{2\alpha +1} \lambda
^{2\alpha +1}d_q \lambda d_q x\\
&=&\int_{0}^{+\infty}g(\lambda )\left(\int_{0}^{+\infty}  f(x)
j_\alpha(\lambda x;q^2) x ^{2\alpha +1}d_q x\right) \lambda
^{2\alpha +1} d_q\lambda\\&=&\int_0^\infty \mathcal{F}_{\alpha
,q}(f)(\lambda) g(\lambda)\lambda^{2\alpha
    +1}d_q\lambda.
\end{eqnarray*}
3)  For  $f\in L_{\alpha ,q}^1 (\R_{q,+})$ such that
$\triangle_{\alpha,q} f\in L_{\alpha ,q}^1 (\R_{q,+})$,  let
$\widetilde{f}$ be the even function defined on $\R_q$  whose  $f$
is its restriction on $\R_{q,+}$. We have
$\widetilde{\triangle_{\alpha,q} f}$=$\triangle_{\alpha,q}
\widetilde{f}$  and   \begin{eqnarray}\mathcal{F}_{\alpha
,q}(\triangle_{\alpha,q} f)(\lambda)&=&c_{\alpha,q}\int_0^\infty
(\triangle_{\alpha,q} f)(x)j_\alpha(x \lambda;q^2)x^{2\alpha
+1}d_qx\\&=&\frac{c_{\alpha,q}}{2}\int_{-\infty}^\infty
(\triangle_{\alpha,q} \widetilde{f})(x)j_\alpha(x
\lambda;q^2)|x|^{2\alpha +1}d_qx.
\end{eqnarray}
So, Proposition \ref{pj} and two $q$-integrations by parts give
the
  result.\\
  4) The result follows from Proposition \ref{pj}.
\endproof
\begin{propo}\label{Finv}
If $f \in L_{\alpha ,q}^1 (\R_{q,+}),$ then $$ \forall x \in
\R_{q,+},~~ ~~ f(x) = c_{\alpha ,q} \ds\int_0^\infty
\mathcal{F}_{\alpha ,q} (f)(\lambda)j_\alpha(\lambda
x;q^2)\lambda^{2\alpha +1}d_q\lambda .$$
\end{propo}
\proof The result follows from the relation (\ref{ij}),
Proposition \ref{ort} and the Fubini's theorem.
\endproof
\begin{theorem}
1) \underline{Plancherel formula} \\
For all $f \in \mathcal{D}_{\ast,q} (\R_q)$, we have \begin
{equation} \label{pbessel}
\|\mathcal{F}_{\alpha,q}(f)\|_{2,\alpha,q}=\|f\|_{2,\alpha, q}.
\end {equation} 2) \underline{Plancherel theorem} \\
The $q$-Bessel transform can be uniquely extended   to an
isometric isomorphism on $L_{\alpha,q}^2(\R_{q,+})$ with
$\mathcal{F}_{\alpha ,q}^{-1} = \mathcal{F}_{\alpha ,q}$.
\end {theorem}
 \proof 1) Let $f \in\mathcal{D}_{\ast,q} (\R_q)$, it is easy to show  that $\mathcal{F}_{\alpha,q}(f)$ is in
 $L_{\alpha,q}^1(\R_{q,+})$. From Proposition
 \ref{Finv}, we have $f = \mathcal{F}_{\alpha,q}(\mathcal{F}_{\alpha,q}(f)),
 $ so using the relation (\ref{sym}), we obtain
 \begin{eqnarray*}
\|f\|_{2,\alpha,q}^2&=&\int_0^\infty f(x)\overline{f}(x)x^{2\alpha
+1}d_qx=\int_0^\infty
\F_{\alpha,q}(\F_{\alpha,q}f)(x)\overline{f}(x)x^{2\alpha
+1}d_qx\\ &=&\int_0^\infty
\F_{\alpha,q}(f)(x)\overline{\F_{\alpha,q}(f)}(x)x^{2\alpha+1}d_qx=\|\F_{\alpha,q}(f)\|_{2,\alpha,q}^2
.
 \end{eqnarray*}
 2) The result follows from 1), Proposition \ref{Finv} and the
 density of $\mathcal{ D}_{\ast,q}(\R_q)$ in  $L_{\alpha,q}^2(\R_{q,+})$.
 \endproof
 \begin{defin}
 For $\ds \alpha
> -\frac{1}{2}$, the $q$-Riemann-Liouville operator  $R_{\alpha ,q } $
is defined for $f \in \mathcal{E}_{\ast ,q}(\R_q)$ by
\begin{equation}\label{R} R_{\alpha ,q }(f)(x) =  \frac{1}{2} C(\alpha ;q^2)
\ds\int_{-1}^1W_\alpha (t;q^2) f(xt) d_qt.
\end{equation}
The $q$-Weyl  operator  is defined for $f \in \mathcal{D}_{\ast
,q}(\R_q)$ by
\begin{equation}\label{tR} ^tR_{\alpha ,q }(f)(t) =   \frac{(1+q)^{-\alpha
+\frac{1}{2}}}{\Gamma _{q^2}(\alpha +\frac{1}{2})} \int_{ q \mid t
\mid }^{+\infty}W_\alpha \left(\frac{t}{x};q^2\right) f(x)
x^{2\alpha}  d_qx.
\end{equation}
\end{defin}

In the end of this section, we shall give some useful properties
of these two operators. First, by simple calculus, one can easily
prove that for $f \in \mathcal{E}_{\ast ,q}(\R_q)$ and  $g \in
\mathcal{D}_{\ast ,q}(\R_q)$, we have
 \begin{equation}\label{relation} \frac{c_{\alpha ,q}}{2} \ds\int_{-\infty}^\infty R_{\alpha ,q
}(f)(x)g(x) |x|^{2\alpha +1}d_qx = K \ds\int_{-\infty}^\infty f(t)
^tR_{\alpha ,q }(g)(t) d_qt. \end{equation} Next, using the
relation (\ref{jmeh}), we obtain
\begin{equation}\label{relation2}
j_\alpha(.;q^2)=R_{\alpha,q}\left(e(-i.;q^2)\right).
\end{equation}
\begin{lemme}\label{cont}
The operator $R_{\alpha ,q }$ is continuous  from
$\mathcal{E}_{*,q}(\R_q)$ into itself.
\end{lemme}
\proof Let $f$ be in $\mathcal{E}_{\ast ,q}(\R_q).$ The function
$x\longmapsto R_{\alpha ,q }(f)(x)$ is an even function on $\R_q.$\\
By $q$-derivation under the $q$-integral sign, we deduce that for
all $n\in \N$,\\
$$\partial _q^n R_{\alpha ,q }(f)(x)= \frac{1}{2} C(\alpha ;q^2)
\ds\int_{-1}^1W_\alpha (t;q^2)t^n (\partial _q^n f)(xt) d_qt.$$
Then,   $$\forall a \geq 0 , \forall n \in \N , P_{n,a}(R_{\alpha
,q }(f)) \leq P_{n,a}(f)<\infty.$$ This relation together with the
Lebesgue theorem proves that $R_{\alpha ,q }(f)$ belongs to
$\mathcal{E}_{\ast ,q}(\R_q)$ and it shows that  the operator
$R_{\alpha ,q }$ is continuous from $\mathcal{E}_{\ast ,q}(\R_q)$
into itself.~~~~~~~~
\endproof
  Using the previous lemma and making a proof as in Theorems 3 and 4
of \cite{BZ}, we obtain the following result.

\begin{theorem} The $q$-Riemann-Liouville  operator $ R_{\alpha ,q } $ is a topological  isomorphism from $\mathcal{E}_{\ast
,q}(\R_q)$ onto itself and it transmutes the operators
$\Delta_{\alpha,q}$ and $\partial_q^2$ in  the following sense
 \begin{equation}\label{trans1}
\Delta_{\alpha,q}R_{\alpha ,q }= R_{\alpha ,q }\partial_q^2.
 \end{equation}
\end{theorem}

\begin{theorem} The $q$-Weyl  operator $ ^tR_{\alpha ,q } $ is an  isomorphism from $\mathcal{D}_{\ast
,q}(\R_q)$ onto itself,  it transmutes the operators
$\Delta_{\alpha,q}$ and $\partial_q^2$ in  the following sense
 \begin{equation}\label{trans2}
^tR_{\alpha ,q }\Delta_{\alpha,q}= \partial_q^2(^tR_{\alpha ,q })
 \end{equation}
and  for $f \in \mathcal{D}_{\ast ,q}(\R_q)$, we have
\begin{equation}\label{relaBF}\mathcal{F}_{\alpha
 ,q} (f)= \left( ^tR_{\alpha ,q
 }(f)\right)\widehat{~~}(.;q^2).\end{equation}
\end{theorem}
\proof The first part of the result can be proved as Proposition 3
of \cite{BZ} page 158.\\
The relation (\ref{relaBF}) is a consequence of the relations
(\ref{relation}) and (\ref{relation2}).\\
Let us now, prove the relation (\ref{trans2}). Let $g\in
\mathcal{D}_{\ast ,q}(\R_q)$. For all $f\in\mathcal{D}_{\ast
,q}(\R_q)$, we have, using the $q$-integration by parts theorem,
the relations (\ref{relation}) and (\ref{trans1}),
\begin{eqnarray*}
&~~&K\int_{-\infty}^\infty\partial_q^2\left(^tR_{\alpha,q}g\right)(x)f(x)d_qx=
K\int_{-\infty}^\infty
 \left(^tR_{\alpha,q}g\right)(x)\partial_q^2f(x)d_qx\\ &=&\frac{c_{\alpha,q}}{2}\int_{-\infty}^\infty
g(x)R_{\alpha,q}\partial_q^2f(x)|x|^{2\alpha+1}d_qx=
\frac{c_{\alpha,q}}{2}\int_{-\infty}^\infty
g(x)\Delta_{\alpha,q}R_{\alpha,q}f(x)|x|^{2\alpha+1}d_qx\\&=&-\frac{c_{\alpha,q}}{2}\int_{-\infty}^\infty
\partial_qg(x)\partial_q(R_{\alpha,q}f)(x)|x|^{2\alpha+1}d_qx=\frac{c_{\alpha,q}}{2}\int_{-\infty}^\infty
\Delta_{\alpha,q}g(x)R_{\alpha,q}f(x)|x|^{2\alpha+1}d_qx\\&=&K\int_{-\infty}^\infty
{^t}R_{\alpha,q}(\Delta_{\alpha,q}g)(x)f(x)d_qx.
\end{eqnarray*}
\section{The $q$-Dunkl operator and its eigenfunctions} For $\ds \alpha \geq -\frac{1}{2}$, consider the operators:\\
 \begin{equation}H_{\alpha , q}: f = f_e + f_o \longmapsto f_e +
 q^{2\alpha +1} f_o
\end{equation}
and
\begin{equation}\Lambda _{\alpha , q}(f)(x) = \partial _q \left[H_{\alpha , q}(
f)\right](x) + [2\alpha +1]_q \frac{f(x)- f(-x)}{2x}.
\end{equation}

It is easy to see that for a differentiable function $f$, the
$q$-Dunkl operator $\Lambda _{\alpha , q}(f)$ tends, as $q$ tends
to 1, to the classical Dunkl operator $\Lambda _{\alpha}(f)$
given by (\ref{D1}).\\
In the case $\ds \alpha = -\frac{1}{2}$, $\Lambda _{\alpha , q}$
reduces to the $q^2$-analogue differential
operator $\partial _q$.\\
Some properties of the $q$-Dunkl operator $\Lambda _{\alpha , q}$
are given in the following proposition.
\begin{propo}.\\
i) If $f$ is odd then $\Lambda _{\alpha , q}(f)(x) = q^{2\alpha
+1}
\partial _q f(x) + [2\alpha +1]_q  \ds \frac{f(x)}{x}$ and
if $f$ is even then $\Lambda _{\alpha , q}(f)(x) =
\partial _q f(x) $.\\
ii) If $f$ and $g$ are of the same parity, then
$$\ds\int_{-\infty}^{+\infty} \Lambda _{\alpha , q}(f)(x)
g(x)|x|^{2\alpha +1} d_qx = 0.$$
 iii) For all $f$ and $g$ such that $\ds\int_{-\infty}^{+\infty} \Lambda _{\alpha , q}(f)(x)
g(x)|x|^{2\alpha +1} d_qx$ exists, we have
\begin{equation}\label{labla}\ds\int_{-\infty}^{+\infty} \Lambda _{\alpha , q}(f)(x)
g(x)|x|^{2\alpha +1} d_qx = - \ds\int_{-\infty}^{+\infty} \Lambda
_{\alpha , q}(g)(x) f(x)|x|^{2\alpha +1} d_qx.\end{equation}
 iv) The operator $\Lambda _{\alpha , q}$ lives
 $\mathcal{E}_q(\R_q)$, $\mathcal{S}_q(\R_q)$ and
 $\mathcal{D}_q(\R_q)$ invariant.
\end{propo}
\begin{proof}
i) is a direct consequence of the definition of $\Lambda _{\alpha
, q}$.\\
ii) follows from the properties of the $q$-integrals and the fact
that $\Lambda _{\alpha , q}$ change the parity of functions.\\
iii) From ii) we have the result when $f$ and $g$ are of the same parity.\\
Now, suppose that $f$ is even and $g$ is odd. Using Lemma
\ref{ppar}, the property i) of $\Lambda _{\alpha , q}$ and the
properties of the $q^2$-analogue differential operator $\partial
_q$ we obtain
\begin{eqnarray*}
\ds\int_{-\infty}^{+\infty} \Lambda _{\alpha , q}(f)(x)
g(x)|x|^{2\alpha +1} d_qx &=& \ds\int_{-\infty}^{+\infty}
\partial _q(f)(x) g(x)|x|^{2\alpha +1} d_qx\\& =&
-\ds\int_{-\infty}^{+\infty}f(x) \partial _q \left
[g(x)|x|^{2\alpha +1} \right ] d_qx \\&=&
-\ds\int_{-\infty}^{+\infty}f(x)\left [ q^{2\alpha +1}
\partial _q g(x) + [2\alpha +1]_q  \ds \frac{g(x)}{x} \right ]|x|^{2\alpha +1} d_qx\\& =& -
\ds\int_{-\infty}^{+\infty} f(x)\Lambda _{\alpha , q}(g)(x)
|x|^{2\alpha +1} d_qx.
\end{eqnarray*}
iv) follows from the facts that  for $f\in \mathcal{E}_q(\R_q)$,
$$\Lambda _{\alpha , q}(f)(x) = \partial _q \left[H_{\alpha , q}(
f)\right](x) + \frac{[2\alpha
+1]_q}{2}\int_{-1}^{1}\partial_q(f)(xt)d_qt$$ and for $f\in
\mathcal{S}_q(\R_q)$,
\begin{eqnarray*}
\Lambda _{\alpha , q}(f)(x) &=& \partial _q \left[H_{\alpha , q}(
f)\right](x) + [2\alpha
+1]_q\int_{0}^{1}\partial_q(f_o)(xt)d_qt\\&=& \partial _q
\left[H_{\alpha , q}( f)\right](x) - [2\alpha
+1]_q\int_{1}^{\infty}\partial_q(f_o)(xt)d_qt.
\end{eqnarray*}

\end{proof}
Let us now introduce the eigenfunctions of the $q$-Dunkl operator.
\begin{theorem}
For $\lambda \in \C$, the $q$-differential-difference equation:
\begin{equation}\label{prob}\left \{
\begin{array}{cc}
\Lambda _{\alpha , q}(f)& = i \lambda f \\
f(0) & =1\quad\\
\end{array}\right.
\end{equation}
has as unique solution, the function
\begin{equation}\label{psi}
\psi _\lambda ^{\alpha ,q}: x\longmapsto  j_\alpha(\lambda x ;q^2)
+\ds \frac{i\lambda x}{[2\alpha +2]_q }  j_{\alpha +1 }(\lambda x
;q^2).
\end{equation}
\end{theorem}
\begin{proof}
Let $f= f_e + f_o.$  The problem (\ref{prob}) is equivalent to the
system

$\left\{ \begin{array}{cc}
                               \partial _q f_e(x)
 +  q^{2\alpha +1}\partial _q f_o(x) + [2\alpha +1]_q \ds  \frac{f_o(x)}{x}& =i\lambda f_e (x)+ i\lambda f_o (x)\\
                              f_e(0)=1,& \\
                             \end{array}\right.$

 which is equivalent to\\\\ $\left \{\begin{array}{cc}
                             \partial _q f_e(x) = i\lambda f_o(x)& \\
                           q^{2\alpha +1}\partial _q^2 f_e(x) + [2\alpha +1]_q \ds \frac{ \partial _q f_e (x)}{x} & =-\lambda ^2 f_e(x)\\
                            f_e(0)=1. &  \\
                        \end{array}\right.$

Now, using Proposition \ref{pj} and the relation (\ref{dqj}), we
obtain\\ $\left\{
\begin{array}{ccc}
 f_e(x)=  j_\alpha(\lambda x ;q^2)&  \\
  f_o(x)= \ds \frac{1}{i\lambda} \partial _q (j_\alpha(\lambda x ;q^2))& =\ds\frac{i\lambda x}{[2\alpha +2]_q}j_{\alpha +1}(\lambda x ;q^2). \\
\end{array}
\right.$\\ Finally, for $\lambda \in \C$,\\ $\psi _\lambda
^{\alpha ,q}( x)= f(x) = j_\alpha(\lambda x ;q^2) +
\ds\frac{i\lambda x}{[2\alpha +2]_q } j_{\alpha +1 }(\lambda x
;q^2)$.
\end{proof}
The function $\psi _\lambda ^{\alpha ,q}( x),$~~ called $q$-Dunkl
kernel has an unique extention to $\C \times \C$~~ and verifies
the following properties.
\begin{propo}\label{psilabla}

1) $\Lambda _{\alpha , q}\psi _\lambda ^{\alpha ,q} = i\lambda
\psi_\lambda ^{\alpha ,q}. $\\

 2) $\ds \psi _\lambda^{\alpha ,q}(x) = \psi _{x}
^{\alpha ,q}(\lambda)$,  $\ds \psi _{a\lambda}^{\alpha ,q}(x) =
\psi _{\lambda} ^{\alpha ,q}(a x)$ \quad and \quad  $
\overline{\psi _\lambda^{\alpha ,q}(x)} = \psi _{-\lambda}
^{\alpha ,q}(x)$, for $\lambda ,x \in \R$ and $a\in \C$.\\

3) If ~~$\ds \alpha = -\frac{1}{2},$ then $\psi _\lambda ^{\alpha
,q}(x) = e(i\lambda x ; q^2).$\\
For $\ds \alpha > -\frac{1}{2} ,~~\psi _\lambda ^{\alpha ,q}$ has
the following $q$-integral representation of Mehler type
\begin{equation}\label{rym} \psi _\lambda ^{\alpha ,q}(x)= \frac{1}{2}
C(\alpha ;q^2) \int_{-1}^1W_\alpha (t;q^2)(1+t) e(i\lambda xt;q^2)
d_qt,
\end{equation}
 where $C(\alpha ;q^2)$ and $W_\alpha
(t;q^2)$ are given respectively by (\ref{cq}) and (\ref{w}).\\
4) For all $n\in \N $ we have
\begin{equation}\label{majdkpsi}
\mid \partial _q^n \psi _\lambda ^{\alpha ,q}(x)\mid \leq
\frac{4\mid \lambda \mid ^n}{(q;q)_\infty},  ~~ \forall \lambda ,x
\in \R_q.
\end{equation}
In particular for all $\lambda\in \R_q$, $\psi _\lambda ^{\alpha
,q}$ is bounded on $\R_q$ and we have
\begin{equation}\label{psib} \mid  \psi _\lambda ^{\alpha ,q}(x)\mid \leq
\frac{4}{(q;q)_\infty}, ~~ \forall x \in \R_q.
\end{equation}
5) For all $\lambda\in \R_q$, $  \psi _\lambda ^{\alpha ,q} \in
 \mathcal{S}_{q}(\R_q)$.
\end{propo}
\proof 1) and 2) are immediate consequences    of the definition
of $\psi _\lambda
^{\alpha ,q}.$\\
3) If $\ds \alpha = -\frac{1}{2}$ then the relations (\ref{jcos}),
(\ref{jsin}) and (\ref{exp}) give the result.\\
If $\ds \alpha > -\frac{1}{2}$, using the definition of $\psi
_\lambda ^{\alpha ,q}$, the parity of the function
$j_\alpha(.;q^2)$ and the relations (\ref{jmeh}) and (\ref{dqj}),
we obtain
\begin{eqnarray*}
  \psi _\lambda ^{\alpha ,q}(x) &=
  & j_\alpha(\lambda x ;q^2)+ \ds \frac{1}{i\lambda}\partial _q \ds(j_\alpha(\lambda x
  ;q^2))\\
  &=& \ds \frac{C(\alpha ;q^2)}{2}  \ds\int_{-1}^1W_\alpha (t;q^2) e(i\lambda xt;q^2)
d_qt + \frac{1}{i}\frac{C(\alpha ;q^2)}{2} \ds\int_{-1}^1W_\alpha
(t;q^2)it e(i\lambda xt;q^2) d_qt,
\end{eqnarray*}
which achieves the proof.\\
4) By induction on $n$ we prove that \\
$\partial _q ^n \psi _\lambda ^{\alpha ,q}(x)= \ds \frac{C(\alpha
;q^2)}{2} (i\lambda )^n \ds\int_{-1}^1W_\alpha (t;q^2)(1+t)t^n
e(i\lambda xt;q^2) d_qt.$ \\ So,  the fact that $|e(i x;q^2)| \leq
\ds \frac{2}{(q;q)_\infty}$ gives the result. \\
5) The result follows from Lemma \ref{asympj}, the relation
(\ref{dqj}) and the properties of $\partial _q$.
\endproof
The function $\psi _\lambda ^{\alpha ,q}$ verifies the following
orthogonality relation.
\begin{propo}\label{ortpsi} For all $x,y \in \R_q$, we have
\begin{equation}\label{psiorto}
    \ds\int_{-\infty}^{+\infty}\psi _\lambda ^{\alpha ,q}(x)\overline
    {\psi _\lambda ^{\alpha ,q}(y)}|\lambda |^{2\alpha
    +1}dq\lambda = \ds\frac{4(1+q)^{2\alpha} \Gamma _{q^2}^2 (\alpha +1)
    \delta _{x,y}}{(1-q)|xy|^{\alpha +1}}.
\end{equation}
\end{propo}
\begin{proof}
Let $x,y \in \R_q$, the use of the relation (\ref{jort}) and the
properties of the $q$-Jackson's integral lead
to{\small\begin{eqnarray*} &~~&\int_{-\infty}^{+\infty}\psi
_\lambda ^{\alpha ,q}(x)\overline
    {\psi _\lambda ^{\alpha ,q}(y)}|\lambda |^{2\alpha
    +1}dq\lambda \\&=& \int_{-\infty}^{+\infty} j_\alpha(\lambda x ;q^2)
     j_\alpha(\lambda y ;q^2)|\lambda |^{2\alpha +1}dq\lambda
     +\frac{xy}{[2\alpha +2]_q^2}\int_{-\infty}^{+\infty} j_{\alpha +1}(\lambda x ;q^2)
     j_{\alpha +1}(\lambda y ;q^2)|\lambda |^{2\alpha
     +3}dq\lambda \\& = &\frac{2(1+q)^{2\alpha} \Gamma _{q^2}^2 (\alpha +1)
      \delta _{|x|,|y|}}{(1-q)|xy|^{\alpha +1}} + \frac{2xy(1+q)^{2\alpha +2} \Gamma _{q^2}^2
       (\alpha +2)\delta _{|x|,|y|}}{[2\alpha
       +2]_q^2(1-q)|xy|^{\alpha +2}}\\& = & \frac{2(1+q)^{2\alpha} \Gamma _{q^2}^2 (\alpha +1)
      \delta _{|x|,|y|}}{(1-q)|xy|^{\alpha +1}}(1 + sgn(xy))= \frac{4(1+q)^{2\alpha}
      \Gamma _{q^2}^2 (\alpha +1) \delta _{x,y}}{(1-q)|xy|^{\alpha
      +1}}.
\end{eqnarray*}}

\end{proof}
\section{ $q-$Dunkl intertwining operator}
 \begin{defin}We define
the $q-$Dunkl intertwining operator $V_\alpha $ on
$\mathcal{E}_q(\R_q)$ by
\begin{equation}\label{D}
    \forall x \in \R_q , V_{\alpha ,q} (f)(x) = \ds \frac{C(\alpha
;q^2)}{2}  \ds\int_{-1}^1W_\alpha (t;q^2)(1+t)f(xt) d_qt,
\end{equation}
where $C(\alpha ;q^2)$ and $W_\alpha (t;q^2)$ are given by
(\ref{cq}) and (\ref{w}) respectively.
\end{defin}
\begin{theorem}\label{vexp}
We have \\
i) $V_{\alpha ,q} (e(-i\lambda x;q^2)) = \psi _{-\lambda} ^{\alpha ,q}(x)$, $\lambda , x\in \R_q$.\\
ii) $V_{\alpha ,q}$ verifies the following  transmutation relation
\begin{eqnarray}\label{vlabla}
  \Lambda _{\alpha , q}V_{\alpha ,q} (f) = V_{\alpha ,q} (\partial
  _qf),\qquad
   V_{\alpha ,q} (f)(0)  = f(0).
\end{eqnarray}
\end{theorem}
\begin{proof}
i) follows from the relation (\ref{rym}).\\
ii) Let $f= f_o + f_e \in \mathcal{E}_{q}(\R_q) $, we have on the one hand\\
$V_{\alpha ,q} (\partial _qf)(x) =  \ds \frac{C(\alpha ;q^2)}{2}
\ds\int_{-1}^1W_\alpha (t;q^2)\partial _q f_o(xt) d_qt +  \ds
\frac{C(\alpha ;q^2)}{2}  \ds\int_{-1}^1W_\alpha (t;q^2)t\partial
_q f_e(xt) d_qt.$\\ On the other hand, we have
\begin{eqnarray*}\Lambda _{\alpha , q}V_{\alpha ,q} (f) (x) &=&
\frac{C(\alpha ;q^2)}{2} \int_{-1}^1W_\alpha (t;q^2)t\partial _q
f_e(xt) d_qt + \frac{q^{2\alpha +1}C(\alpha ;q^2)}{2}
\int_{-1}^1W_\alpha (t;q^2)t^2 \partial _q f_o(xt) d_qt \\&+&
 \frac{[2\alpha +1]_qC(\alpha ;q^2)}{2x}  \ds\int_{-1}^1W_\alpha
(t;q^2)t f_o(xt) d_qt.
\end{eqnarray*}
Now, using a $q$-integration by parts and the facts that
$$ \partial_q\left[(1-q^2t^2)W_\alpha(qt;q^2)\right]=-[2\alpha
+1]_qtW_\alpha(t;q^2)$$ and
$$(1-q^2t^2)W_\alpha (qt;q^2) =(1-t^2q^{2\alpha +1}) W_\alpha
(t;q^2),$$ we get
\begin{eqnarray*}[2\alpha +1]_q \ds \frac{C(\alpha
;q^2)}{2x} \int_{-1}^1W_\alpha (t;q^2)t f_o(xt) d_qt &=&
\frac{C(\alpha ;q^2)}{2}\ds\int_{-1}^1(1-q^2t^2)W_\alpha (qt;q^2)
\partial _q f_o(xt) d_qt\\ &=&\frac{C(\alpha ;q^2)}{2}\ds\int_{-1}^1(1-t^2q^{2\alpha +1}) W_\alpha
(t;q^2)
\partial _q f_o(xt) d_qt,\end{eqnarray*}
which completes the proof.
\end{proof}
\begin{theorem}\label{vrq}
For all $f\in \mathcal{E}_q(\R_q)$, we have
\begin{equation}\label{vr}
    \forall x \in \R_q , V_{\alpha ,q} (f)(x) = R_{\alpha ,q} (f_e)(x) +
    \partial _q  R_{\alpha ,q} I_q(f_o)(x),
\end{equation}
where $R_{\alpha ,q}$ is given by (\ref{R}) and $I_q $ is the operator given by \\
$$ \forall x \in \R_q , I_q(f_o)(x) = \ds\int_0^{|qx|}f_o(t)
d_qt.$$
\end{theorem}
\begin{proof}From the definitions of the $q-$Dunkl intertwining
 and the $q$-Riemann-Liouville operators, we have
\begin{eqnarray*}
  V_{\alpha ,q} (f)(x) &=& \ds \frac{C(\alpha ;q^2)}{2}
\ds\int_{-1}^1W_\alpha (t;q^2)(1+t) (f_o(xt) + f_e(xt)) d_qt \\
  &=& \ds \frac{C(\alpha ;q^2)}{2}
\ds\int_{-1}^1W_\alpha (t;q^2) f_e(xt) d_qt +  \ds \frac{C(\alpha
;q^2)}{2}
\ds\int_{-1}^1W_\alpha (t;q^2)t f_o(xt)  d_qt. \\
   &=& R_{\alpha ,q} (f_e)(x)+ \ds \frac{C(\alpha
;q^2)}{2} \ds\int_{-1}^1W_\alpha (t;q^2)t f_o(xt)  d_qt.
\end{eqnarray*}
On the other hand, by $q$-derivation under the $q$-integral sign
and the fact that\\ $\partial _q (I_qf_o) = f_o$,  we obtain
{\small\begin{eqnarray*}
    \partial _q \left[R_{\alpha ,q} I_q(f_o)\right](x)=  \frac{C(\alpha
;q^2)}{2} \ds\int_{-1}^1W_\alpha (t;q^2)t \partial _q (I_qf_o)(xt)
d_qt= \frac{C(\alpha ;q^2)}{2} \ds\int_{-1}^1W_\alpha (t;q^2)t
f_o(xt) d_qt.
\end{eqnarray*}}
This gives the result.
\end{proof}
\begin{theorem}\label{ines}
The transform $V_{\alpha ,q} $ is an isomorphism from
$\mathcal{E}_q(\R_q)$ onto itself, its  inverse transform is given
by
\begin{equation}\label{vinvers}
     \forall x \in \R_q , V_{\alpha ,q}^{-1} (f)(x) = R_{\alpha ,q}^{-1} (f_e)(x) +
    \partial _q \left( R_{\alpha ,q}^{-1} I_q(f_o)\right)(x),
\end{equation}
where $R_{\alpha ,q}^{-1}$ is the inverse transform of $R_{\alpha
,q}$.
\end{theorem}
\begin{proof}
Let H be the operator defined on $\mathcal{E}_q(\R_q)$ by\\
$$H(f) = R_{\alpha ,q}^{-1} (f_e) +
    \partial _q  (R_{\alpha ,q}^{-1} I_q(f_o)).$$
 We have
$ V_{\alpha ,q} (f) = R_{\alpha ,q} (f_e) +
    \partial _q  \left(R_{\alpha ,q} I_q(f_o)\right)$ ,  $  R_{\alpha ,q} (f_e)$ is even
    and
    $\partial _q ( R_{\alpha ,q} I_q(f_o))$ is odd, then
\begin{eqnarray*}
  H  V_{\alpha ,q} (f)&=& R_{\alpha ,q}^{-1}R_{\alpha ,q}f_e +
  \partial _qR_{\alpha ,q}^{-1}I_q(\partial _q  R_{\alpha ,q} I_q(f_o))\\
   &=& f_e+ \partial _qR_{\alpha ,q}^{-1}I_q(\partial _q  R_{\alpha ,q}
   I_q(f_o)).
\end{eqnarray*}
Using the fact that for $\varphi\in\mathcal{E}_{*,q}(\R_q)$, $\ds
I_q(\partial _q \varphi)(x) = \varphi (x) - \lim_{t\rightarrow
0}\varphi (t)$, we obtain
  $$I_q(\partial _q  R_{\alpha ,q} I_q(f_o)) =  R_{\alpha ,q} I_q(f_o).   $$
  So,$$ R_{\alpha ,q}^{-1}I_q(\partial _q  R_{\alpha ,q}
I_q(f_o))=I_q(f_o)
  $$ and $$ \partial _qR_{\alpha ,q}^{-1}I_q(\partial _q  R_{\alpha ,q} I_q(f_o))=
  \partial _q I_q(f_0) =
 f_0.$$
Thus, $$ H  V_{\alpha ,q} (f) = f_e + f_o = f.$$ With the same
technique, we prove that $ V_{\alpha ,q}H (f) = f.$
\end{proof}
\begin{defin} For $f \in \mathcal{D}_q (\R_q)$ and  $\ds  \alpha
>-\frac{1}{2}$, we define the $q$-transpose of $V_{\alpha ,q}$
by
\begin{equation}\label{tV}
   (^tV_{\alpha ,q})(f)(t) = M_{\alpha,q}   \ds\int_{|x |\geq q|t|}
  W_\alpha \left(\frac{t}{x};q^2\right)\left(1+ \frac{t}{x}\right)f(x) \frac{|x|^{2\alpha +1}}{x}d_qx
,\end{equation} where $W_\alpha(.;q^2)$ is given by (\ref{w}) and
\begin{equation}\label{MtV}
    M_{\alpha,q} =  \frac{(1+q)^{-\alpha +\frac{1}{2}}}{2\Gamma _{q^2}(\alpha
    +\frac{1}{2})}.
\end{equation}\end{defin}
Note that by simple computation, we obtain for $f \in
\mathcal{E}_{q}(\R_q)$ and $g \in \mathcal{D}_{q}(\R_q)$
\begin{equation}\label{vtv}
\ds\frac{c_{\alpha ,q}}{2} \ds\int_{-\infty}^{+\infty}V_{\alpha
,q}(f)(x)g(x)|x|^{2\alpha +1}d_qx = K
\ds\int_{-\infty}^{+\infty}f(t)(^tV_{\alpha ,q})(g)(t) d_qt.
\end{equation}

\begin{propo}
For $f \in \mathcal{D}_{q}(\R_q)$, we have
\begin{equation}\label{tvlabla}
   \partial _q(^tV_{\alpha ,q})(f) = (^tV_{\alpha ,q})(\Lambda _{\alpha ,
   q})(f).
\end{equation}
\end{propo}
\begin{proof}
Using a $q$-integration by parts and the relations  (\ref{vtv}),
(\ref{vlabla}) and (\ref{labla}), we get for all $f \in
\mathcal{D}_{q}(\R_q)$ and $g \in \mathcal{E}_{q}(\R_q)$,
\begin{eqnarray*}
  K \ds\int_{-\infty}^{+\infty}g(x)\partial _q (^tV_{\alpha ,q})f(x) d_qx&=& -K
   \ds\int_{-\infty}^{+\infty} \partial _q g(x) (^tV_{\alpha ,q})f(x) d_qx\\
  &=& -\frac{ c_{\alpha ,q}}{2}\ds\int_{-\infty}^{+\infty}V_{\alpha ,q}(\partial _q g)(x)
  f(x)|x|^{2\alpha +1} d_qx \\
  &=& -\frac{ c_{\alpha ,q}}{2}\ds\int_{-\infty}^{+\infty} \Lambda
_{\alpha , q}(V_{\alpha ,q}g)(x) f(x)|x|^{2\alpha +1} d_qx\\
  &=&  \frac{c_{\alpha ,q}}{2}\ds\int_{-\infty}^{+\infty} V_{\alpha ,q}(g)(x) \Lambda
_{\alpha , q}f(x)|x|^{2\alpha +1} d_qx \\
  &=&  K \ds\int_{-\infty}^{+\infty}g(x)(^tV_{\alpha ,q})(\Lambda
_{\alpha , q}f)(x)d_qx.
\end{eqnarray*} As $g$ is arbitrary in $\mathcal{E}_{q}(\R_q)$, we obtain the result.
\end{proof}
\begin{theorem}
For $f \in \mathcal{D}_{q}(\R_q)$, we have
\begin{equation}\label{vr}
    \forall x \in \R_q,  (^tV_{\alpha ,q}) (f)(x) =  (^tR_{\alpha ,q}) (f_e)(x) +
    \partial _q  \left[^tR_{\alpha ,q} J_q(f_o)\right](x),
\end{equation}
where $ ^tR_{\alpha ,q}$ is given by (\ref{tR}) and $J_q$ is the
operator defined by
$$J_q(f_o)(x) = \ds \int_{-\infty}^{qx} f_o(x) d_qx.$$

\end{theorem}
\begin{proof}
Let $f,g \in \mathcal{D}_q(\R_q)$, using Theorem \ref{vrq}, the
relation (\ref{relation}) and  a $q$-integration by parts, we obtain\\
$\ds\frac{c_{\alpha ,q}}{2} \ds\int_{-\infty}^{+\infty}V_{\alpha
,q}(g)(x)f(x)|x|^{2\alpha +1}d_qx = \ds\frac{c_{\alpha ,q}}{2}
\ds\int_{-\infty}^{+\infty}\left[ R_{\alpha ,q} (g_e)(x) +
    \partial _q  R_{\alpha ,q} I_q(g_o)(x)\right]f(x)|x|^{2\alpha
    +1}d_qx$\\
$= \ds\frac{c_{\alpha ,q}}{2} \ds\int_{-\infty}^{+\infty}
R_{\alpha ,q} (g_e)(x).f_e(x).|x|^{2\alpha +1}d_qx +
\ds\frac{c_{\alpha ,q}}{2} \ds\int_{-\infty}^{+\infty}\partial _q
R_{\alpha ,q} I_q(g_o)(x).f_o(x).|x|^{2\alpha +1}d_qx$\\
$=K \ds\int_{-\infty}^{+\infty} (^tR_{\alpha ,q})
(f_e)(x).g_e(x)d_qx -\ds\frac{c_{\alpha ,q}}{2}
\ds\int_{-\infty}^{+\infty} R_{\alpha ,q} I_q(g_o)(x).\partial _q
\left[f_o(x).|x|^{2\alpha +1}\right]d_qx.$\\
It is easily seen that the map $J_q$ is bijective from
$\mathcal{D}_q^\ast(\R_q)$ onto $\mathcal{D}_{\ast ,q}(\R_q)$ and
$J_q^{-1}=\partial _q$, where $\mathcal{D}_q^\ast(\R_q)$ is the
subspace of $\mathcal{D}_q(\R_q)$ constituted of odd functions.\\
Hence, by writing $f_o = \partial _q J_qf_o$ and by making use
 of (\ref{trans1}) and (\ref{relation}) we get
{\small\begin{eqnarray*} &~~&\frac{c_{\alpha ,q}}{2}
\ds\int_{-\infty}^{+\infty} R_{\alpha ,q} I_q(g_o)(x).\partial _q
\left[f_o(x).|x|^{2\alpha +1}\right]d_qx\\& =& \ds\frac{c_{\alpha
,q}}{2} \int_{-\infty}^{+\infty} R_{\alpha ,q}
I_q(g_o)(x).\ds\frac{1}{|x|^{2\alpha +1}}\partial _q
\left[|x|^{2\alpha +1}\partial _q J_qf_o(x)\right]|x|^{2\alpha
+1}d_qx\\&=& \frac{c_{\alpha ,q}}{2} \int_{-\infty}^{+\infty}
R_{\alpha ,q} I_q(g_o)(x).\Delta_{\alpha,q}J_qf_o(x).|x|^{2\alpha
+1}d_qx= K\ds\int_{-\infty}^{+\infty} I_q(g_o)(x).^tR_{\alpha
,q}\Delta_{\alpha,q}J_qf_o(x)d_qx\\&=& K\int_{-\infty}^{+\infty}
I_q(g_o)(x).\partial _q^2(^tR_{\alpha
,q})J_qf_o(x)d_qx=-K\int_{-\infty}^{+\infty}\partial_q
I_q(g_o)(x).\partial _q(^tR_{\alpha ,q})J_qf_o(x)d_qx.
\end{eqnarray*}} Since $\partial_q I_q(g_o)(x)=
g_o(x)$,
then\\
$\ds\frac{c_{\alpha ,q}}{2} \ds\int_{-\infty}^{+\infty}V_{\alpha
,q}(g)(x)f(x)|x|^{2\alpha +1}d_qx = K\ds\int_{-\infty}^{+\infty}
g(x)\left[ (^tR_{\alpha ,q})f_e(x) +\partial _q(^tR_{\alpha
,q})J_qf_o(x)\right]d_qx.$\\
As $g $ is arbitrary in $\mathcal{D}_q(\R_q)$, this relation when
combined with (\ref{vtv}) gives  the result.
\end{proof}
\begin{theorem} The transform $(^tV_{\alpha ,q}) $ is an isomorphism from
$\mathcal{D}_q(\R_q)$ onto itself, its  inverse transform is given
by
\begin{equation}\label{tvinvers}
     \forall x \in \R_q , (^tV_{\alpha ,q})^{-1} (f)(x) = (^tR_{\alpha ,q})^{-1} (f_e)
     (x) +
    \partial _q  \left[(^tR_{\alpha ,q})^{-1} J_q(f_o)\right](x),
\end{equation}
where $(^tR_{\alpha ,q})^{-1}$ is the inverse transform of ~\quad
$^tR_{\alpha ,q}$.
\end{theorem}
\begin{proof}
Taking account of the relation $J_q \partial_q f(x) = f(x) $ for
all $f\in \mathcal{D}_{\ast ,q}(\R_q)$ and proceeding as in
Theorem \ref{ines} we obtain the result.
\end{proof}
\section{$q$-Dunkl transform}
 \begin{defin} Define the $q$-Dunkl
transform for $f \in L_{\alpha , q}^1 (\R_q)$ by
\begin{equation}\label{FD}
    F_D^{\alpha ,q}(f)(\lambda) = \ds\frac{c_{\alpha ,q}}{2}\ds\int_{-\infty}^{+\infty}
    f(x) \psi _{-\lambda} ^{\alpha ,q}(x).|x|^{2\alpha +1} d_qx,
\end{equation}
where $c_{\alpha ,q}$ is given by (\ref{c}).
\end{defin}
{\bf Remarks :}\\1) It is easy to see that in the even case
$F_D^{\alpha,q}$ reduces to the $q$-Bessel Fourier transform given
by (\ref{FB}) and in the case $\ds \alpha=-\frac{1}{2}$, it
reduces to the $q^2$-analogue Fourier transform given by
(\ref{fou}). \\
2)  Letting $q\uparrow 1$ subject to the condition (\ref{q}),
gives, at least formally, the classical Bessel-Dunkl transform.\\
 Some properties of the $q$-Dunkl transform are given in the
 following proposition.
 \begin{propo}
 i) If $f \in L_{\alpha ,q}^1 (\R_q)$ then $ F_D^{\alpha ,q}(f)\in L_{q}^\infty(\R_q)$,
  \begin{equation}
\|F_D^{\alpha ,q}(f)\|_{\infty,q} \leq \frac{2c_{\alpha
,q}}{(q;q)_\infty} \| f \|_{1,\alpha ,q}\end{equation} and
\begin{equation*}
\lim_{\lambda\rightarrow \infty}F_D^{\alpha ,q}(f)(\lambda)=0.
\end{equation*}
 ii) For $f \in
L_{\alpha ,q}^1 (\R_q),$\begin{equation}\label{FDlab}
 F_D^{\alpha ,q}(\Lambda _{\alpha , q}f)(\lambda )= i\lambda
F_D^{\alpha ,q}(f)(\lambda ).\end{equation}
 iii) For $f,g \in
L_{\alpha ,q}^1 (\R_q),$
 \begin{eqnarray}\label{symD}
 \ds\int_{-\infty}^{+\infty}F_D^{\alpha ,q}(f)(\lambda )g(\lambda
)|\lambda |^{2\alpha +1}d_q \lambda &=&
\ds\int_{-\infty}^{+\infty}f(x )F_D^{\alpha ,q}(g)(x )|x
|^{2\alpha +1}d_q x.
\end{eqnarray}
\end{propo}
\begin{proof}
i) Follows  from the definition of $F_D^{\alpha
,q}(f)$, the Lebesgue theorem  and the fact that $\ds |\psi _{-\lambda} ^{\alpha ,q}(x)|\leq \frac{4}{(q;q)_\infty},$ for all $\lambda,~~x\in\R_q.$\\
ii) Using the relation (\ref{labla}) and Proposition
\ref{psilabla},
we obtain the result.\\
iii) Let  $f,g \in L_{\alpha ,q}^1 (\R_q).$ \\ Since for all
$\lambda ,x \in \R_q $,we have $\ds \mid  \psi _\lambda ^{\alpha
,q}(x)\mid \leq \frac{4}{(q;q)_\infty}$,  then
\begin{eqnarray*}
  \ds\int_{-\infty}^{+\infty}\ds\int_{-\infty}^{+\infty} \mid f(x)
g(\lambda ) \psi _\lambda ^{\alpha ,q}(x)| |x |^{2\alpha +1}
|\lambda |^{2\alpha +1}d_q x d_q \lambda  & \leq &
\frac{4}{(q;q)_\infty} \| f\|_{1,\alpha ,q} \| g\|_{1,\alpha ,q}.
\end{eqnarray*}
So, by the Fubini's theorem, we can exchange the order of the
$q$-integrals, which gives the result.
\end{proof}

\begin{theorem}\label{invD}
For all $f\in L_{\alpha ,q}^1 (\R_q)$, we have
\begin{equation}\begin{split}\label{Dinv}
\forall x\in \R_q,\quad f(x)& =\frac{c_{\alpha
,q}}{2}\int_{-\infty}^{+\infty}
    F_D^{\alpha ,q}(f)(\lambda) \psi _{\lambda} ^{\alpha ,q}(x).|\lambda|^{2\alpha +1}
    d_q\lambda\\ &= \overline{F_D^{\alpha ,q}(\overline{F_D^{\alpha ,q}(f)})} (x) .\end{split}\end{equation}
\end{theorem}
\proof Let $f\in L_{\alpha ,q}^1 (\R_q)$ and $ x\in \R_q $. Since
for all $\lambda, t  \in \R_q $, we have\\ $\ds \mid  \psi
_\lambda ^{\alpha ,q}(t)\mid \leq \frac{4}{(q;q)_\infty}$,  and
$\lambda\mapsto \psi _\lambda ^{\alpha ,q}(x)$ is in $
\mathcal{S}_{q}(\R_q)$,
then\\
{\small
\begin{eqnarray*}
\int_{-\infty}^\infty\int_{-\infty}^\infty | f(t)\psi _{-\lambda}
^{\alpha ,q}(t)\psi _{\lambda} ^{\alpha
,q}(x)||t\lambda|^{2\alpha+1}d_qtd_q\lambda &\leq&
\frac{4}{(q;q)_\infty}\int_{-\infty}^\infty\int_{-\infty}^\infty |
f(t)||\psi _{\lambda} ^{\alpha
,q}(x)||t\lambda|^{2\alpha+1}d_qtd_q\lambda\\&=&\frac{4}{(q;q)_\infty}\|f\|_{1,\alpha,q}\|\psi
_{x} ^{\alpha ,q}(\centerdot)\|_{1,\alpha,q}. \end{eqnarray*}}
Hence, by the Fubini's theorem, we can exchange the order of the
$q$-integrals and by  Proposition \ref{ortpsi}, we obtain
\begin{eqnarray*}&~~&\frac{c_{\alpha ,q}}{2}\int_{-\infty}^\infty F_D^{\alpha,
q}(f)(\lambda)\psi _{\lambda} ^{\alpha
,q}(x)|\lambda|^{2\alpha+1}d_q\lambda\\&=&\left(\frac{c_{\alpha
,q}}{2}\right)^2\int_{-\infty}^\infty
f(t)\left(\int_{-\infty}^\infty \psi _{-\lambda} ^{\alpha
,q}(t)\psi _{\lambda} ^{\alpha
,q}(x)|\lambda|^{2\alpha+1}d_q\lambda\right)|t|^{2\alpha+1}d_qt=f(x).
\end{eqnarray*}
The second equality is a direct consequence of the definition of
the $q$-Dunkl transform, Proposition \ref{psilabla} and the
definition of the $q$-Jackson integral.
\endproof
\begin{theorem}
i) \underline{Plancherel formula} \\
For $\ds \alpha\geq -1/2$, the $q$-Dunkl transform $F_D^{\alpha,
q}$ is an isomorphism from $\mathcal{S}_q (\R_q)$ onto itself.
Moreover, for all $f \in \mathcal{S}_q (\R_q)$, we have \begin
{equation} \label{pldun}
\|F_D^{\alpha,q}(f)\|_{2,\alpha,q}=\|f\|_{2,\alpha, q}.
\end {equation}
ii) \underline{Plancherel theorem} \\
The $q$-Dunkl transform can be uniquely extended  to an isometric
isomorphism on $L_{\alpha,q}^2(\R_q)$. ~~Its  inverse transform
${(F_D^{\alpha ,q})}^{-1}$ is given by :
\begin{equation}
    {(F_D^{\alpha ,q})}^{-1}(f)(x) =  \frac{c_{\alpha ,q}}{2}\ds\int_{-\infty}^{+\infty}
    f(\lambda) \psi _{\lambda} ^{\alpha ,q}(x).|\lambda|^{2\alpha +1}
    d_q\lambda =F_D^{\alpha ,q}(f)(-x).
\end{equation}
\end {theorem}
\proof
 i) From Theorem \ref{invD}, to prove the first part of
 i) it suffices to prove that $F_D^{\alpha,
q}$ lives $\mathcal{S}_q (\R_q)$ invariant. Moreover, from the
definition of $\mathcal{S}_q (\R_q)$ and the properties of the
operator $\partial_q$ (Lemma \ref{ld}), one can easily see that
$\mathcal{S}_q (\R_q)$ is also the set of all function defined on
$\R_q$, such that for all $k, l\in\N$, we have$$
\sup_{x\in\R_q}\left|\partial_q^k
\left(x^lf(x)\right)\right|<\infty\quad {\rm and}\quad
\lim_{x\rightarrow 0} \partial_q^k f(x)\quad {\rm exists}.$$ Now,
let $f\in \mathcal{S}_q (\R_q)$
 and $k, l\in\N$. On the one hand, using the notation $\ds
 \Lambda_{\alpha,q}^0f=f$ and\\ $\ds
 \Lambda_{\alpha,q}^{n+1}f=\Lambda_{\alpha,q}(\Lambda_{\alpha,q}^nf)$,
 $n\in\N$,
 we obtain from the properties of the operator $
 \Lambda_{\alpha,q}$ that for all $n\in \N$,
 $\Lambda_{\alpha,q}^nf\in \mathcal{S}_q (\R_q)\subset L_{\beta,q}^1(\R_q)
 $ for all $\ds \beta \geq -1/2$.\\
On the other hand,  from the relation (\ref{FDlab}), we have
\begin{eqnarray*}
\lambda^l F_D^{\alpha, q}(f)(\lambda)&=&(-i)^lF_D^{\alpha,
q}(\Lambda_{\alpha,q}^lf)(\lambda)\\&=&
(-i)^l\frac{c_{\alpha,q}}{2}\int_{-\infty}^\infty
\Lambda_{\alpha,q}^lf (x)\psi^{\alpha,q}_{-\lambda}(x)|x|^{2\alpha
+1}d_qx.
\end{eqnarray*}
So, using the relation (\ref{majdkpsi}), we obtain
\begin{eqnarray*}|\partial_q^k(\lambda^l F_D^{\alpha,
q}(f)(\lambda))|&=&\left|
(-i)^l\frac{c_{\alpha,q}}{2}\int_{-\infty}^\infty
\Lambda_{\alpha,q}^lf
(x)\partial_q^k\psi^{\alpha,q}_{-x}(\lambda)|x|^{2\alpha
+1}d_qx\right|\\&\leq&
\frac{2c_{\alpha,q}}{(q;q)_{\infty}}\int_{-\infty}^\infty
|\Lambda_{\alpha,q}^lf (x)||x|^{2\alpha
+k+1}d_qx<\infty.\end{eqnarray*}
 This together with the Lebesgue theorem  prove that  $F_D^{\alpha, q}(f)$ belongs to $\mathcal{S}_q
 (\R_q)$.\\
  By Theorem \ref{invD}, we
 deduce that $F_D^{\alpha, q}$ is an isomorphism of $\mathcal{S}_q
 (\R_q)$ onto itself and for $f\in\mathcal{S}_q (\R_q)$, we have
 $\ds (F_D^{\alpha ,q})^{-1}(f)(x)=F_D^{\alpha ,q}(f)(-x),\quad  x\in
 \R_q.$\\
Finally, the Plancherel formula (\ref{pldun}) is a direct
consequence of the  second equality in Theorem \ref{invD} and the
relation (\ref{symD}).

ii) The result follows from i), Theorem \ref{invD} and the
 density of $\mathcal{ S}_q (\R_q)$ in
 $L_{\alpha,q}^2(\R_{q})$.
\endproof
\begin{theorem}\label{brahim}
 The $q$-Dunkl transform and the $q^2$-analogue Fourier transform are linked by
\begin{equation}\label{7} \forall f \in \mathcal{D}_{q} (\R_q),\quad
    F_D^{\alpha ,q}(f) = \left[ ^t V_{\alpha ,q} (f)\right]\widehat{~~}(.;q^2).
\end{equation}
\end{theorem}
\proof Using the relation (\ref{vtv}) and Theorem \ref{vexp}, we
obtain for $f\in \mathcal{D}_{q} (\R_q),$
\begin{eqnarray*}
  \left[ ^t V_{\alpha ,q} (f)\right]\widehat{~~}(\lambda )&=&  K \ds\int_{-\infty}^{+\infty}
  (^t V_{\alpha ,q}) (f)(t) e(-i\lambda t;q^2) d_qt\\&=& \frac{c_{\alpha ,q}}{2}
\ds\int_{-\infty}^{+\infty} V_{\alpha ,q}
  (e(-i\lambda x;q^2))f(x)|x|^{2\alpha +1}d_qx \\
 &=&  \frac{c_{\alpha ,q}}{2}\int_{-\infty}^{+\infty}
    f(x) \psi _{-\lambda} ^{\alpha ,q}(x).|x|^{2\alpha +1} d_qx \\
   &=&  F_D^{\alpha ,q}(f)(\lambda ).
\end{eqnarray*}
\endproof

\end{document}